\documentclass[a4paper, 11pt]{article}
\usepackage{amsmath,amssymb,amscd}
\usepackage{amsthm}
\usepackage[all]{xy}
\usepackage[dvipdfmx]{graphicx,xcolor}
\usepackage{tikz}

\usepackage[top=30truemm,bottom=30truemm,left=25truemm,right=25truemm]{geometry}

\theoremstyle{definition}
\newtheorem{lemma}{Lemma}[section]
\newtheorem{definition}[lemma]{Definition}
\newtheorem{theorem}[lemma]{Theorem}
\newtheorem{example}[lemma]{Example}
\newtheorem{proposition}[lemma]{Proposition}
\newtheorem{corollary}[lemma]{Corollary}
\newtheorem{remark}[lemma]{Remark}

\newcommand{\F}{{\mathbb F}}
\newcommand{\im}{\operatorname{Im}}
\newcommand{\Ker}{\operatorname{Ker}}
\newcommand{\Spec}{\operatorname{Spec}}

\title{\bf  On specializations of minimal $p$-divisible groups}
\author{Nobuhiro Higuchi\thanks{Graduate School of Environment and Information Sciences, Yokohama National University, 
79-1 Tokiwadai, Hodogaya-ku, Yokohama 240-8501 JAPAN.
E-mail: \texttt{higuchi-nobuhiro-tc@ynu.jp}}
\ and 
Shushi Harashita\thanks{Graduate School of Environment and Information Sciences, 
Yokohama National University,
79-1 Tokiwadai, Hodogaya-ku, Yokohama 240-8501 JAPAN.
E-mail: \texttt{harasita@ynu.ac.jp}}}

\begin{document}

\maketitle

\begin{abstract}

In this paper, for any pair $(\zeta, \xi)$ of Newton polygons 
with $\zeta \prec \xi$, we construct a concrete specialization 
from the minimal $p$-divisible group of $\xi$ 
to the minimal $p$-divisible group of $\zeta$ by a beautiful induction. 
This in particular gives the affirmative answer to the unpolarized analogue of 
a question by Oort on the boundaries of central streams,
and gives another proof of the dimension formula of the central leaves in the unpolarized case.

\end{abstract}


\section{Introduction}
Let $p$ be a rational prime.
In this paper,
by a $p$-divisible group
we mean a Barsotti-Tate group in algebraic and arithmetic geometry,
i.e., an inductive limit of finite algebraic group schemes
having some properties.
The precise definition of $p$-divisible groups
will be reviewed at the begining of Section \ref{Section2_Background}.
We study $p$-divisible groups in characteristic $p$.
By the Dieudonn\'e-Manin classification (cf.~\cite{Manin}),
the isogeny classes of $p$-divisible groups over an algebraically closed field
in characteristic $p$ are classified by Newton polygons,
see Definition \ref{DefNewtonPolygon} for the definition of Newton polygons.

Let $\xi$ be a Newton polygon.
Among $p$-divisible groups having Newton polygon $\xi$, 
there is a special $p$-divisible group which is called minimal.
We denote it by $H(\xi)$.
The main reference for minimal $p$-divisible groups is Oort \cite{oort2}.
As proposed in the latter part of \cite[Question~6.10]{oort}, 
$H(\zeta)$ was expected to appear as a specialization of $H(\xi)$ 
for $\zeta \prec \xi$, see Definition \ref{DefNewtonPolygon}
for the notation about Newton polygons.
Our main theorem (Theorem~\ref{theorem1}) implies
that the expectation is true:

\begin{corollary}\label{corollary1}
If $\zeta \prec \xi$, then
$H(\zeta)$ appears as a special fiber of
a $p$-divisible group having $H(\xi)$
as geometric generic fiber.
\end{corollary}
The converse of this corollary is also true, which is a consequence of
Grothendieck-Katz \cite{Katz}, Theorem 2.3.1.

Now several proofs of this corollary have been known
(cf. \cite{VW}, Proposition 1.8 and \cite{SNP}, Corollary 5.1),
but the authors could not find any known result which implies
Theorem~\ref{theorem1}.
Also the proof of Theorem~\ref{theorem1} 
has an advantage of giving a very concrete
construction of such specializations.
For example it would be interesting to study relations between the construction
and that in the proof in \cite[Corollary 5.1]{SNP}.
The method of our proof is based on the idea of \cite{harashita1}, where
the second auther proved
the similar result in the polarized case with application to the theory of
stratifications of the moduli space of principally polarized abelian varieties. Remark that, combining Corollary \ref{corollary1} and \cite{harashita2}, Theorem 1.1, one can prove the unpolarized analogue of Oort's conjecture \cite[6.9]{oort}, see \cite{SNP}, Corollary 5.2.
From Theorem~\ref{theorem1}, we can also give a new proof of
the dimension formula of central leaves in the unpolarized case,
see Corollary~\ref{CorDimensionFormula}.

The essential case for the proof of the main theorem is that 
$\xi$ consists of two segments and $\zeta \prec \xi$ is saturated. 
The case that the slopes of $\xi$, say $\lambda_2 < \lambda_1$, 
satisfy $\lambda_2<\lambda_1 \leq 1/2$ or
$1/2 \leq \lambda_2 < \lambda_1$
has been proved in \cite{harashita1}, 8.4.
This paper confirms that the above theorem holds in the remaining case 
$\lambda_2 < 1/2 < \lambda_1$.

This paper is organized as follows.
In Section 2, we recall the definition of 
$p$-divisible groups and 
truncated Dieudonn\'e modules of level one (abbreviated as ${\rm DM_1}$), 
which are Dieudonn\'e modules of $p$-kernels of $p$-divisible groups.
We also recall the definitions of 
(${\rm DM_1}$-)simple ${\rm DM_1}$'s, their direct sums and minimal ${\rm DM_1}$'s. 
In Section 3, we recall some facts on specializations of ${\rm DM_1}$'s.
In Section 4, we treat some combinatorics on Newton polygons. 
In Section 5, we state our main results. 
In Section 6, we give a proof of the main theorem (Theorem~\ref{theorem1})
whose beautiful induction would hopefully influence some future works. 

\section{Background}\label{Section2_Background}
Let $p$ be a rational prime, and $h$ a non-negative integer.
Let $S$ be a scheme. 
A {\it $p$-divisible group (Barsotti-Tate group)} of height $h$ over $S$
is an inductive
system $X=(G_\nu,i_\nu)$ ($\nu=0,1,2,\cdots$),
where $G_\nu$ is a
finite locally free commutative group scheme over $S$ of order $p^{\nu h}$
and, for each $\nu$ the sequence of commutative group schemes
\[
\begin{CD}
0 @>>> G_\nu @>i_\nu>> G_{\nu+1} @>p^\nu>> G_{\nu+1}
\end{CD}
\]
is exact, that is to say, $G_\nu$ is identified via $i_\nu$ with
$\Ker(p^\nu: G_{\nu+1} \to G_{\nu+1})$.

For a $p$-divisible group $X=(G_\nu,i_\nu)$ over $S$ and
for a morphism $T \to S$ of schemes,
we have a $p$-divisible group $X_T$ over $T$
defined by $(G_\nu\times_S T,i_\nu\times {\rm id})$.
In particular for a closed point $s=\Spec(k)\to S$
the $p$-divisible group $X_s$ over $k$ is called
the {\it fiber} of $X$ over $s$.
A $p$-divisible group obtained as the fiber of $X$ over a closed point
is said to be {\it a special fiber of $X$}.

Let $S$ be an $\F_p$-scheme.
For any $S$-scheme $T$, 
let ${\rm Frob}$ be the absolute frobenius on $T$
and ${\rm Fr}: T \to T^{(p)}$ be the relative frobenius.

A {\it truncated Bartotti-Tate group of level one} (${\rm BT_1}$)
over $S$ is a finite locally free commutative group scheme over $S$
such that $\im({\rm Ver}: G^{(p)}\to G) = \Ker({\rm Fr}:G\to G^{(p)})$
and $\Ker({\rm Ver}: G^{(p)}\to G) = \im({\rm Fr}:G\to G^{(p)})$,
where $G^{(p)} = G\times_{S,{\rm Frob}}S$ and
${\rm Ver}$ is the Verschebung on $G$.
The $p$-kernel $X[p] := {\rm Ker} (p : X \rightarrow X)$
of a $p$-divisible group $X$ is a ${\rm BT_1}$,
and any ${\rm BT_1}$ over an algebraically closed field
is the $p$-kernel of a $p$-divisible group.

Let $k$ be a perfect field of characteristic $p$. 
Let $W(k)$ denote the ring of Witt vectors with coefficients in $k$.
Let $\sigma$ be the frobenius map on $k$,
and use the same symbol $\sigma$ for the frobenius map on $W(k)$.
A {\it Dieudonn\'e module} over $k$ is a finite $W(k)$-module $M$
equipped with $\sigma$-linear homomorphism $F: M \to M$
and $\sigma^{-1}$-linear homomorphism $V: M \to M$
such that $F\circ V$ and $V\circ F$ are equal to the multiplication by $p$.
In this paper we use the covariant Dieudonn\'e theory.
It says that there is an equivalence from the category of
$p$-divisible group (resp. $p$-torsion finite commutative group schemes) over $k$ to that of
Dieudonn\'e modules which is free (resp. is of finite length) as $W(k)$-modules.
In this paper, Dieudonn\'e modules corresponding to ${\rm BT_1}$'s
via the Dieudonn\'e functor
is called ${\rm DM_1}$'s. The precise definition of them is as follows.
\begin{definition}\label{DefDM1}

A truncated Dieudonn\'e module of level one (abbreviated as ${\rm DM_1}$) 
over $k$ of height $h$ is a triple $(N, F, V)$ consisting of 
a $k$-vector space $N$ of dimension $h$, 
a $\sigma$-linear homomorphism $F:N \rightarrow N$ and 
a ${\sigma}^{-1}$-linear homomorphism $V:N \rightarrow N$ satisfying 
$\mathop{\rm Ker} F=\mathop{\rm Im} V$ and 
$\mathop{\rm Im} F=\mathop{\rm Ker} V$. 

\end{definition}


When $k$ is an algebraically closed field,
the following theorem is known
(cf. Kraft \cite{kraft}, Oort \cite{oort1} and Moonen-Wedhorn \cite{moonen-wedhorn}). 

\begin{theorem}\label{theorem2}

There exists a bijection:

  \begin{center}
    
    $\{0,1\}^h\longleftrightarrow \{{\rm DM_1} \ $over$\ k\ $of height$\ h\}/\cong.$
  
  \end{center}

\end{theorem}

Before we recall the bijection in this theorem,
we give a remark.
\begin{remark}
Giving a ${\rm DM_1}$ over $k$
is equivalent to giving an $F$-zip over $k$
with support contained in $\{0,1\}$
with the terminology in \cite{moonen-wedhorn},
where Moonen-Wedhorn used some subsets of the Weyl group of $GL_h$ (in this case) as classifying data of them over $k$.
Using the Weyl group as classifying data is quite natural, but
we shall not use the structure of the Weyl group in this paper.
We here use $\{0,1\}^h$, as classifying data of ${\rm DM_1}$'s,
which has an advantage
when we treat decompositions of ${\rm DM_1}$'s into direct summands
often considered in this paper.
\end{remark}

We identify $\{0,1\}^h$ with the set of maps from $\{1,\ldots, h\}$ to $\{0,1\}$, i.e., with the set of sequences of $0$ and $1$ with length $h$.
Let $A$ be an element of $\{0,1\}^h$, 
and let $\delta:\{1,\ldots, h\}\ni i \mapsto \delta_{i}\in \{0,1\}$
be the map corresponding to $A$. Then we express $A$ as 
the sequence $\delta_{1}\delta_{2}\cdots \delta_{h}$.
The bijection in the theorem is defined by the following. 
To a sequence $A$, we associate a ${\rm DM_1}$ $N=ke_1\oplus \cdots \oplus ke_h$
with $F$ and $V$ defined as follows. 
We define maps $F$, $V$ using $\delta$ as follows. 
\begin{equation}\label{DefOfF}
  Fe_i= \begin{cases}
    e_j\ (j=\# \{t \mid 1\leq t \leq i ;\ \delta_t=0\}) & $if$\ \delta_i=0, \\
    0 & $if$\ \delta_i=1.
  \end{cases}
\end{equation}
Let $u$ (resp. $v$) be the number of 1 (resp. 0) in sequence $A$. 
Let $e_{i_1},\ldots,e_{i_u}\ (i_1<\cdots <i_u)$ be the set of $e_i$ with $\delta_i=1$.
\begin{equation}\label{DefOfV}
  Ve_j= \begin{cases}
    e_{i_k} & $for$\ j>v,\ k=j-v, \\
    0 & $for$\ j\leq v.
  \end{cases}
\end{equation}
One can check that the obtained triple $(N, F, V)$ is a ${\rm DM_1}$. 

To express the Dieudonn\'e module associated with
$A=\delta_1\delta_2\cdots\delta_h$,
the diagram with arrows
\begin{equation}\label{FVDiagma}
\xymatrix{\delta_j &\delta_i, \ar@/_18pt/[l]_F &\delta_{i_k} \ar@/_18pt/[r]^{V^{-1}} &\delta_j}
\end{equation}
is useful, where we used the notation in \eqref{DefOfF} and \eqref{DefOfV}.
	This diagram is called the {\it $(F,V^{-1})$-diagram} of $A$
(or of the ${\rm DM_1}$ associated with $A$).
As an example, let us look at the ${\rm DM_1}$ associated with sequence 10100. 
Let $N=\langle e_1, e_2, e_3, e_4, e_5\rangle$. 
The $(F,V^{-1})$-diagram is
\\
$$\xymatrix{1\ar@/_20pt/[rrr]^{V^{-1}} &0\ar@/_18pt/[l]_F &1\ar@/_20pt/[rr]^{V^{-1}} &0\ar@/_18pt/[ll]_F &0\ar@/_18pt/[ll]_F \\ e_1 & e_2 & e_3 & e_4 & e_5}$$
This means
$$\begin{cases}
  Fe_2=e_1, & Fe_1=0,\\
  Fe_4=e_2, & Fe_3=0,\\
  Fe_5=e_3,
\end{cases}\ \ \ \ \ \ \ \ \ \ \ \ \ 
\begin{cases}
  Ve_1=0, & Ve_4=e_1,\\
  Ve_2=0, & Ve_5=e_3,\\
  Ve_3=0.
\end{cases}$$
It follows from $\mathop{\rm Ker}F=\mathop{\rm Im}V=\{e_1, e_3\}$ 
and $\mathop{\rm Im}F=\mathop{\rm Ker}V=\{e_1, e_2, e_3\}$ that $(N, F, V)$ is a ${\rm DM_1}$.

\begin{definition}\label{DefNewtonPolygon}
A {\it Newton polygon} $\xi$ is a finite multiple set of pairs $(m_i,n_i)$
($i=1,\ldots,t$) of coprime non-negative integers. Conventionally we write
$\xi = \sum_{i=1}^t (m_i,n_i)$.
We call $\sum_{i=1}^t n_i$ the {\it dimension} of $\xi$
and $\sum_{i=1}^t (m_i+n_i)$ the {\it height} of $\xi$.
Put $\lambda_i=n_i/(m_i+n_i)$.
We arrange $(m_i,n_i)$ so that
$\lambda_i\ge \lambda_j$ for $i<j$.
We regard $\xi$ as the downward-convex line graph starting at $(0,0)$ and ending at $(h,d)$
with breaking points $\sum_{i > j}(m_i+n_i,n_i)$ for $j=0,\ldots,t$.
For two Newton polygons $\xi$ and $\zeta$ with same ending point,
we say $\zeta \prec \xi$ if any point of $\zeta$ is above or on $\xi$.
We say that $\zeta \prec \xi$ is {\it saturated}
if there is no Newton polygon $\eta$ such that
$\zeta \precneqq \eta \precneqq \xi$.
\end{definition}

To each Newton polygon, we assoiate a ${\rm DM_1}$,
which is called minimal:

\begin{definition}
Let $m$ and $n$ be non-negative integers with gcd($m,n$)=1. 
To $(m,n)$, we associate a ${\rm DM_1}$ which corresponds to the sequence 
$$A_{m,n} = \underbrace{1\ 1\cdots1}_{m}\underbrace{0\cdots0}_{n}$$ 
by Theorem~\ref{theorem2}, 
and we write it as $N_{m,n}$. Such ${\rm DM_1}$'s are called {\it simple}
(more precisely should be called {\it ${\rm DM_1}$-simple}). 
A {\it minimal} ${\rm DM_1}$ is
the direct sum
\[
N_{\xi} := \bigoplus N_{m_i,n_i}
\] 
for some Newton polygon $\xi = \sum (m_i,n_i)$.
We may frequently identify a ${\rm DM_1}$ $N$ and $A\in \{0,1\}^h$ 
if $N$ is the ${\rm DM_1}$ corresponding to $A$, and 
we write $A \in \{0,1\}^h$ as $A_{\xi}$ 
if $A$ corresponds to $N_{\xi}$.
\end{definition}

This is the notion of Dieudonn\'e modules of
$p$-kernels of minimal $p$-divisible groups.
Let us recall the definition of minimal $p$-divisible groups.
For each coprime pair $(m,n)$ of non-negative integers,
let $H_{m,n}$ be the $p$-divisible group $H_{m,n}$ over $\F_p$ whose
Dieudonn\'e module ${\mathbb D}(H_{m,n})$ is given by
\[
{\mathbb D}(H_{m,n}) = \bigoplus_{i=1}^{m+n} {\mathbb Z}_p e_i
\]
with $F,V$-operations defined by
$Fe_i=e_{i-m}$ and $Ve_i=e_{i-n}$, where
$e_i$ for non-positive $i$ is inductively defined by $e_i=p e_{i+m+n}$.
We set
\[
H(\xi) = \bigoplus H_{m_i,n_i}
\]
for each Newton polygon $\xi = \sum (m_i,n_i)$.
A {\it minimal} $p$-divisible group is a $p$-divisible group
which is isomorphic over an algebraically closed field 
to $H(\xi)$ for some Newton polygon $\xi$.
Let $k$ be an algebraically closed field.
The Dieudonn\'e module of the $p$-kernel of $H(\xi)_k$
is isomorphic to $N_\xi$.
The main theorem of \cite{oort2} says that
for any $p$-divsible group $X$ over $k$,
if ${\mathbb D}(X[p])\simeq N_\xi$, then $X$ is isomorphic to $H(\xi)_k$.


Let $N_1$ and $N_2$ be two ${\rm DM_1}$'s.
Let $A$ and $B$ be the elements of $\{0,1\}^{h_1}$ and $\{0,1\}^{h_2}$ 
corresponding to $N_1$ and $N_2$ respectively.
Recall how to get the element of $\{0,1\}^{h_1+h_2}$ corresponding to 
the direct sum $N_1\oplus N_2$.
For this, we define a real number $b_{A}(i)$
with $0\le b_{A}(i) \le 1$ for each $i=1,\ldots, h_1$.
To define $b_A(i)$, 
we consider the $(F, V^{-1})$-diagram of $A$.
Running though the arrows in the reverse direction from $\delta^A_i$
we set $b_l=0$ if the $l$-th arrow is $F$
and $b_l=1$ if the $l$-th arrow is $V^{-1}$.
We define $b_A(i)$ to be the binary expansion
\[
b_A(i) = 0.b_1b_2\cdots.
\]
In the similar way, we define $b_{B}(j)$ for $j=1,\ldots, h_2$.
The sequence corresponding to $N_1\oplus N_2$, written as $A\oplus B$,
is obtained by arranging $\delta^A_i$ ($i=1,\ldots,h_1$) and
$\delta^B_j$ ($j=1,\ldots,h_2$) 
in ascending order of their binary expansions, namely
$A\oplus B = \cdots\delta_i^A\cdots\delta_j^B\cdots$ if $b_A(i) < b_B(j)$ and
$\cdots\delta_j^B\cdots\delta_i^A\cdots$ if $b_B(j) < b_A(i)$, and so on.

\begin{example}\label{Example_N35N32}

Here, let us see $N_{3,5}\oplus N_{3,2}$
as an example of direct sum of ${\rm DM_1}$'s.
We write $N_{3,5}$ and $N_{3,2}$ in the following using sequences. \\ 
$$N_{3,5} : \xymatrix@=8pt{1_1\ar@/_20pt/[rrrrr] &1_2\ar@/_20pt/[rrrrr] &1_3\ar@/_20pt/[rrrrr]&0_4\ar@/_20pt/[lll] &0_5\ar@/_20pt/[lll] &0_6\ar@/_20pt/[lll] &0_7\ar@/_20pt/[lll] &0_8\ar@/_20pt/[lll]},\ \ \ \ 
N_{3,2} : \xymatrix@=8pt{ 1_1\ar@/_20pt/[rr] & 1_2\ar@/_20pt/[rr] & 1_3\ar@/_20pt/[rr] &
0_4\ar@/_20pt/[lll] & 0_5\ar@/_20pt/[lll]}$$ 
\\
Let $A\in\{0,1\}^8$ and $B\in\{0,1\}^5$ be the sequences
of $N_{3,5}$ and $N_{3,2}$ respectively. 

Consider $b_A(8)$ for example: we trace vectors
in the reverse direction from $0_8$:
\[
0_8 \xleftarrow[V^{-1}]{} 1_3 \xleftarrow[\ F\ ]{} 0_6 \xleftarrow[V^{-1}]{} 1_1 
\xleftarrow[\ F\ ]{} 0_4 \xleftarrow[\ F\ ]{} 0_7 \xleftarrow[V^{-1}]{} 1_2 
\xleftarrow[\ F\ ]{}0_5 \xleftarrow[\ F\ ]{} 0_8 \xleftarrow[\ V^{-1}\ ]{} \cdots.
\] 
Hence we get $b_{A}(8)=0.10100100\cdots$. Similarly we have
\begin{eqnarray*}
b_A(8)=0.10100100\cdots, && b_B(1)=0.01011\cdots,\\
b_A(7)=0.10010100\cdots, && b_B(2)=0.01101\cdots,\\
b_A(6)=0.10010010\cdots, && b_B(3)=0.10101\cdots,\\
b_A(5)=0.01010010\cdots, && b_B(4)=0.10110\cdots.
\end{eqnarray*}
By the above, we in particular get:
$$b_A(5)<b_B(1)<b_B(2)<b_A(6)<b_A(7)<b_A(8)<b_B(3)<b_B(4).$$
Then the sequence $A_\xi$ corresponding to
$N_{\xi}=N_{3,5}\oplus N_{3,2}$ is
$$A_{\xi}=1^A_1\ \ 1^A_2\ \ 1^A_3\ \ 0^A_4\ \ 
0^A_5\ \ 1^B_1\ \ 1^B_2\ \ 0^A_6\ \ 0^A_7\ \ 
0^A_8\ \ 1^B_3\ \ 0^B_4\ \ 0^B_5,$$ 
where, to avoid confusion, we write each elements $1_i$ (resp. $0_i$) of 
$A$ as $1^A_i$ (resp. $0^A_i$),
and we write each elements $1_j$ (resp. $0_j$) of 
$B$ as $1^B_j$ (resp. $0^B_j$).

\end{example}

More generally we have
\begin{lemma}\label{lem1}
Let $\xi=(m_1,n_1)+(m_2,n_2)$
with $\lambda_2 < 1/2 < \lambda_1$ 
where $\lambda_i = n_i/(m_i+n_i)$.
The sequence associated with $N_\xi=N_{m_1,n_1}\oplus N_{m_2,n_2}$ is given by
$$\underbrace{1^A_1\cdots 1^A_{m_1}}_{m_1}
\underbrace{0^A_{m_1+1}\cdots 0^A_{n_1}}_{n_1-m_1}
\underbrace{1^B_1\cdots 1^B_{n_2}}_{n_2}\ 
\underbrace{0^A_{n_1+1}\cdots 0^A_{m_1+n_1}}_{m_1}
\underbrace{1^B_{n_2+1}\cdots 1^B_{m_2}}_{m_2-n_2}\ 
\underbrace{0^B_{m_2+1}\cdots 0^B_{m_2+n_2}}_{n_2},$$
where $A$ and $B$ are the sequences associated with $N_{m_1,n_1}$ and $N_{m_2,n_2}$ respectively.
\end{lemma}

\begin{proof}
See \cite{harashita1}, Proposition 4.20.
\end{proof}


\section{Specializations}
We introduce the notion of families of ${\rm DM_1}$'s
(but such a family will be also called a ${\rm DM_1}$ simply),
and review some basic facts on specializations of ${\rm DM_1}$'s.

Let $R$ be a commutative ring of characteristic $p>0$.
Let $\sigma:R \rightarrow R$ be the frobenius endomorphism 
defined by $\sigma(a)=a^p$.

\begin{definition}\label{DefFamilyDM1}

A ${\rm DM_1}$ over $R$ of height $h$ is a quintuple 
${\cal N}=({\cal N}, C, D, F, V^{-1})$ 
where
\begin{enumerate}
\item[(1)] ${\cal N}$ is a free $R$-module of rank $h$, 
\item[(2)]$C$ and $D$ are submodules of ${\cal N}$
which are locally direct summands of ${\cal N}$,
\item[(3)]
$F:({\cal N}/C) \otimes_{R,\sigma} R\rightarrow D$ and 
$V^{-1}: C\otimes_{R,\sigma} R\rightarrow {\cal N}/D$
are $R$-linear isomorphisms. 
\end{enumerate}

\end{definition}

\begin{remark}
\begin{enumerate}
\item[(1)] In Moonen and Wedhorn \cite{moonen-wedhorn},
${\rm DM_1}$'s over $R$ here are called
$F$-zips over $R$ with support contained in $\{0,1\}$.
\item[(2)] When $R$ is a perfect field $k$,
to a ${\rm DM_1}$ $(N,V,F)$ over $k$ with the notation in Definition \ref{DefDM1}
we associate a quintuple $(N,VN,FN,F,V^{-1})$, which naturally becomes
a ${\rm DM_1}$ with the notation in Definition \ref{DefFamilyDM1}.
By this association, we identify them.
\end{enumerate}
\end{remark}

Let $k$ be an algebraically closed field of characteristic $p$,
and let $R=k[\![t]\!]$ be the ring of formal power series over $k$. 
For $\mathcal N$ an arbitrary ${\rm DM_1}$ over $R$,
we can consider ${\mathcal N}_k := {\mathcal N}\otimes_R k$, which is a ${\rm DM_1}$ over $k$.
Hence we have the canonical map called {\it specialization}
\begin{center}
$\{{\rm DM_1} \ $over$ \ R\}\longrightarrow \{{\rm DM_1} \ $over$ \ k\}$.
\end{center}
sending ${\cal N}$ to ${\cal N}_k$.

Let $K$ be the fractional field of $R$. 
We also consider ${\cal N}_{\overline K} := {\cal N}\otimes_R {\overline K}$,
which is a ${\rm DM_1}$ over $\overline K$ and is called the geometric generic fiber of $\mathcal N$.

\begin{definition}
Let $A$ and $B$ be elements of $\{0,1\}^h$.
We say $B \preceq A$ if there exists an ${\rm DM_1}$ over $R$ 
such that ${\cal N}_k$ is associated with $B$ and ${\cal N}_{\overline K}$ is associated with $A$.
\end{definition}

Let $A, B$ be the sequences of $N_1, N_2$ respectively. 
We denote  by $A\oplus B$ the sequence corresponding to $N_1\oplus N_2$.
It is obvious that if $B \prec A$, then $B\oplus P \prec A\oplus P$ holds
for $A,\ B \in \{0, 1\}^h$ and $P \in \{0, 1\}^{h'}$.

\begin{definition}
Let $A$ and $B$ be elements of $\{0,1\}^h$.
We say $B < A$ if 
there exist elements $A^{(1)}, \ldots, A^{(\ell-1)}$ of $\{0,1\}^h$
with $A^{(0)}:=B$ and $A^{(\ell)}:=A$ such that
one can write $A^{(i)}=P^{(i)} \oplus Q^{(i)}$, and 
$A^{(i-1)}=P^{(i)} \oplus Q^{(i-1)}$ for $i=1, \ldots, \ell$
for some $P^{(i)}$, where $Q^{(i-1)}$ is constructed by exchanging
an adjacent subsequence ``0 1" in $Q^{(i)}$ for ``1 0". 

\end{definition}

The following is known,
see the proof of \cite[Proposition~11.1]{oort1}.

\begin{proposition}

If $B \leq A$, then $B \preceq A$ holds.

\end{proposition}


We want to know relations between ${\cal N}\otimes_R {\overline K} \ $and$ \ {\cal N}\otimes_R k$. 
The aim of this paper is to show the existence of ${\cal N}$ satisfying 
${\cal N}\otimes_R \overline{K}\cong N_{m_1,n_1}\oplus N_{m_2,n_2}=N_{\xi}$ 
and \ ${\cal N}\otimes_R k \cong N_{m'_1,n'_1}\oplus N_{m'_2,n'_2}\oplus 
\cdots \oplus N_{m'_i,n'_i}=N_{\zeta}$ for any $\zeta \prec \xi$. 
Before dealing with the general case, 
let us see as an example the case of $\zeta=(2,3)+4(1,1)\prec \xi=(3,5)+(3,2)$. 
\begin{example}\label{example1}

We consider $N_{\xi}=N_{3,5}\oplus N_{3,2}$.
As seen in Example~\ref{Example_N35N32}, the sequence corresponding to $N_{\xi}$ is
$$A_{\xi}=1\ 1\ 1\ 0\ \underline{0\ 1}\ 1\ 0\ 0\ 0\ 1\ 0\ 0.$$
Let $A^-_{\xi}$ be the sequence obtained by 
exchanging 0 and 1 of the above underline part.
Then another subsequence $0\ 1$ appears, see the double-underline part below: 
$$A^-_{\xi}=1\ 1\ 1\ \underline{\underline{0\ 1}}\ 0\ 1\ 0\ 0\ 0\ 1\ 0\ 0.$$
Let $A^{--}_{\xi}$ be the sequence obtained by exchanging 0 and 1 of the double-underline part:
\[
A^{--}_{\xi}=1\ 1\ 1\ 1\ 0\ 0\ 1\ 0\ 0\ 0\ 1\ 0\ 0.
\]
The $(F,V^{-1})$-diagram of $N^{--}_{\xi}$ associated with $A^{--}_{\xi}$ is
\\

$$N^{--}_{\xi} : \xymatrix@=10pt{1\ar@/_20pt/[rrrr] &1\ar@/_20pt/[rrrr] &0\ar@/_20pt/[ll] &1\ar@/_20pt/[rrr] &0\ar@/_20pt/[lll] &1\ar@/_20pt/[rr] &0\ar@/_20pt/[llll] &0\ar@/_20pt/[llll]}\ \oplus \ 
\xymatrix@=10pt{1\ar@/_20pt/[rrr] &1\ar@/_20pt/[rrr] &0\ar@/_20pt/[ll] &0\ar@/_20pt/[ll] &0\ar@/_20pt/[ll]}.$$
\\
The right direct summand $1\ 1\ 0\ 0\ 0$ is the sequence of $N_{2,3}$,
and the left direct summand $A' := 1\ 1\ 0\ 1\ 0\ 1\ 0\ 0$ can be specialized to
$B:=1\ 1\ 1\ 1\ 0\ 0\ 0\ 0$ by exchanging repeatedly 0 1 for 1 0.
Note that $B$ is the sequence of $N_{1,1}^{\oplus 4}$.
Hence we get a specialization from $N_\xi$ to
$N_{\zeta}=N_{2,3}\oplus N_{1,1}\oplus N_{1,1}\oplus N_{1,1}\oplus N_{1,1}$.
The figure of  $\xi$ and $\zeta$ is as follows.

\includegraphics[width=180mm]{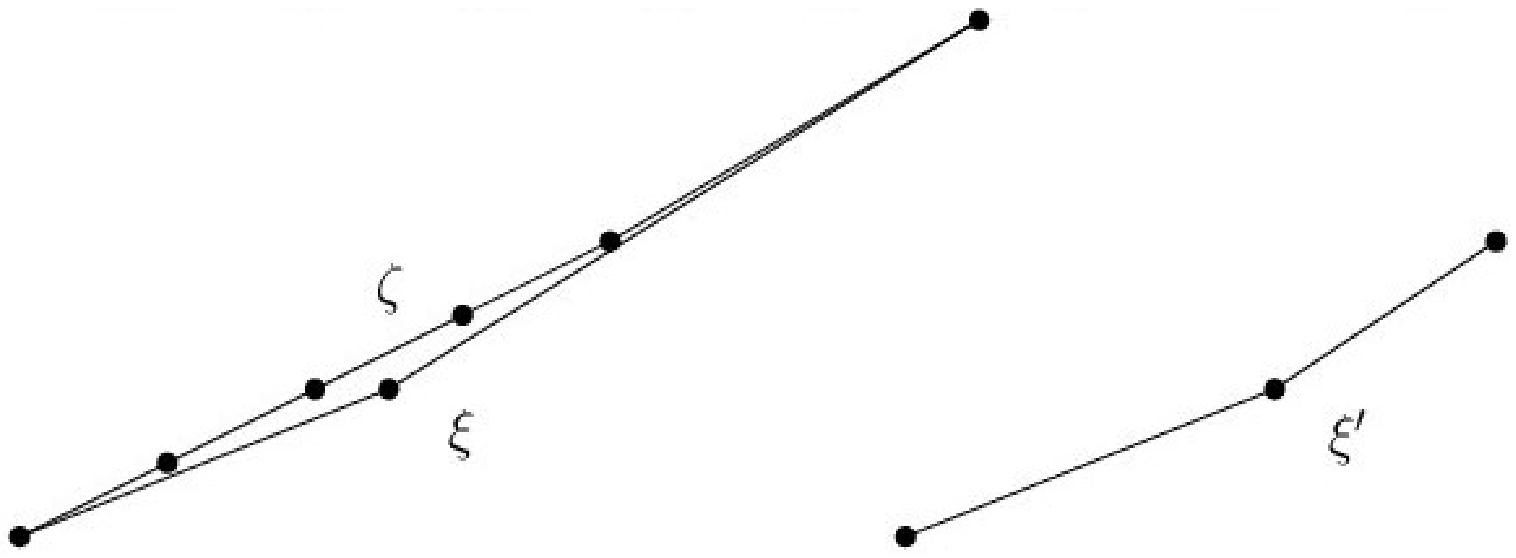}

\noindent Put $\xi' = (1,2) + (3,2)$. Then
\[
A_{\xi'} = 1\ 0\ 0 \oplus 1\ 1\ 1\ 0\ 0 = 1\ \underline{0\ 1}\ 1\ 0\ 1\ 0\ 0.
\]
We define $A_{\xi'}^-$
in the similar way as when we defined $A_{\xi}^-$ from $A_{\xi}$, i.e.,
\[
A_{\xi'}^- = 1\ 1\ 0\ 1\ 0\ 1\ 0\ 0.
\]
We observe 
that $A_{\xi'}^-$ is equal to $A'$ above.
The key step of our proof is to show that such phenomenon always occurs.
As the height of $\xi'$ is less than that of $\xi$,
this allows us to prove the main theorem by induction.
\end{example}

\section{Combinatorics of Newton polygons}
We show some combinatorial facts on Newton polygons,
which will be used later on.

For Newton polygons $\zeta\prec\xi$ of height $h$ we set
\begin{equation}
\displaystyle c(\zeta,\xi)=2\sum^{h}_{i=1}(\zeta(i)-\xi(i)),
\end{equation}
where we regard Newton polygons
(line graphs in the $xy$-plane)
as functions on $\{x\in\mathbb{R}\mid 0\le x\le h\}$.
There is a relationship between $c(\zeta,\xi)$ and
the number of segments of $\zeta$ in a special case:

\begin{proposition}\label{prop4}
Let $\xi$ be a Newton polygon consisting of two segments.
Let $\zeta\prec \xi$ be a saturated pair of Newton polygons.
Then $c(\zeta,\xi)$ is equal to the number of segments of $\zeta$.

\end{proposition}

\begin{proof}
First note that $c(\zeta,\xi)$
is equal to the area 
of the part which is surrounded by $\zeta$ and $\xi$.

We prove the proposition
by induction on the number $c$ of the segments of $\zeta$.
The case of $c=1$ is obvious from Lemma \ref{SaturatednessImpliesAreaOne} below. Assume $c\ge 2$.
Write $\xi =(m_1,n_1)+(m_2,n_2)$ and $\zeta =(m'_1,n'_1)+\cdots+(m'_c,n'_c)$.
We set $\zeta'=(m_2',n_2')+\cdots +(m_c',n_c')$ and
$\xi'=(m_1-m'_1,n_1-n'_1)+(m_2,n_2)$ if $m_1\ge m'_1$ and $n_1\ge n'_1$
or
$\zeta'=(m_1',n_1')+\cdots +(m_{c-1}',n_{c-1}')$ and
$\xi'=(m_1,n_1)+(m_2-m'_c,n_2-n'_c)$ otherwise.
We write a proof only in the former case,
as the same argument works for the latter case.
The figure in the former case is as follows.\\
\includegraphics[width=200mm]{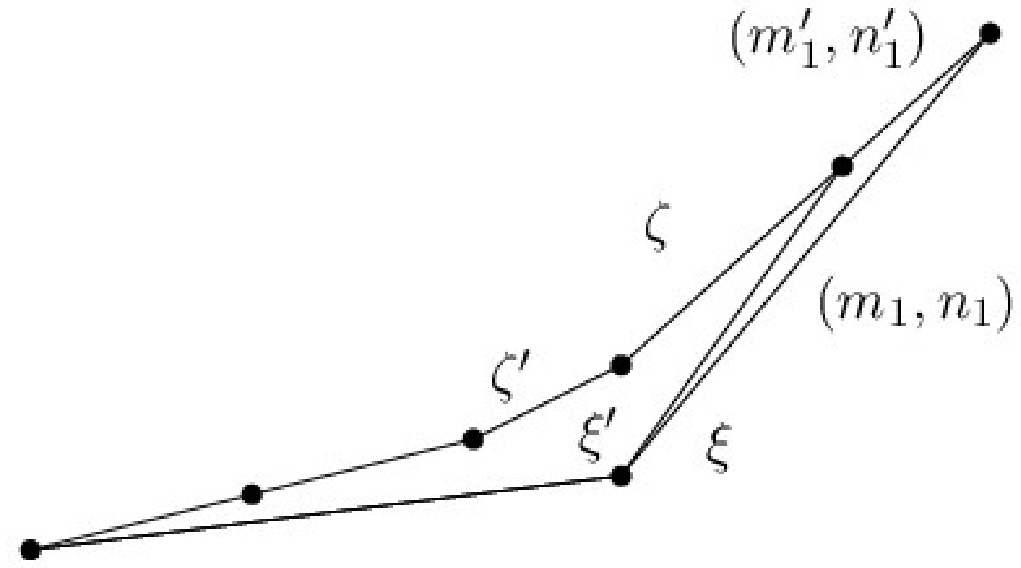}
\\
We have $\zeta'\prec \xi'$ and this is saturated. 
By the hypothesis of induction, 
the number of segments of $\zeta'$ is equal to $c(\zeta',\xi')$, 
and therefore the number of segments of $\zeta$ is equal to $c(\zeta',\xi')+1$. 
Considering the areas we have the following. 

\[
c(\zeta,\xi)
=c(\zeta',\xi') + n_1m'_1-m_1n'_1 
\]
We have to show that $n_1m'_1-m_1n'_1=1$ 
under the assumption that there is no lattice point
in the region surrounded by $(m_1',n_1')$, $\xi$ and $\xi'$.
This follows from the lemma below.
\end{proof}

For vectors $\vec{a}=(a_1,a_2)$ and $\vec{b}=(b_1,b_2)$, 
we set $\langle \vec{a},\vec{b}\rangle=a_1b_2-a_2b_1$.

\begin{lemma}\label{SaturatednessImpliesAreaOne}
Let $\vec{a}=(a_1,a_2)$, $\vec{b}=(b_1,b_2)$ with $a_1, a_2, b_1, b_2\in \mathbb{Z}$ 
and $\langle \vec{a},\vec{b} \rangle >0$. 
Assume that there exists no lattice point in the interior of 
the triangle that is the convex hull of $(0,0)$, $\vec{a}$ and $\vec{b}$.
Then $\langle \vec{a}, \vec{b}\rangle=1$.
\end{lemma}

\begin{proof}
Assume $v:=\langle \vec{a},\vec{b} \rangle \ge 2$,
and lead a contradiction. 
Choose $\vec{x}=(x,y)$ with $\langle \vec{a},\vec{x} \rangle=1$ and $x,y \in \mathbb{Z}$. 
Choose $r\in \mathbb{Z}$, $s,t\in \mathbb{Q}$ such that $\vec{x}+r\vec{a}=s\vec{a}+t\vec{b}$ with $0\le s\le 1$.
Taking $\langle - ,\vec{a}\rangle$ on the both sides, 
we have $\langle \vec{a},\vec{x}\rangle+r\langle \vec{a},\vec{a}\rangle =s\langle \vec{a},\vec{a}\rangle+t\langle \vec{a},\vec{b}\rangle$,
whence $1=tv$. 
We have $\vec{x} + r\vec{a} = s\vec{a} +(1/v)\vec{b}$.
As $ s\in (1/v)\mathbb{Z}$ with $0 \le s < 1$,
we have $s+(1/v)\le 1$.
This means that $\vec{x} + r\vec{a}$ is a lattice point
in the interior of the triangle that is the convex hull of $(0,0)$, 
$\vec{a}$ and $\vec{b}$. This is a contradiction.
\end{proof}

\section{Main results}

Our main theorem is the following.

\begin{theorem}\label{theorem1}

Let $\zeta$ and $\xi$ be Newton polygons of height $h$ with $\zeta \prec \xi$. 
Set $M_0:=N_{\zeta}$ and $M_c:=N_{\xi}$ with $c=c(\zeta,\xi)$. 
Then there exist ${\rm DM_1}$'s $M_1,\ldots, M_{c-1}$ such that $A^{(0)}<A^{(1)}\cdots<A^{(c-1)}<A^{(c)}$, 
where $A^{(i)}\in\{0,1\}^h$ are the sequences corresponding to $M_i$ ($0 \leq i \leq c$).

\end{theorem}

The proof will be given in the next section.
Let us see that
this theorem implies Corollary~\ref{corollary1} in Introduction.
\begin{proof}[Proof of Corollary~\ref{corollary1}.]
It suffices to show the case that
$\zeta$ contains no \'etale segment $(1,0)$.
Indeed write $\zeta = f(1,0) + \zeta'$ and $\xi = f(1,0) + \xi'$ 
for $f \in \mathbb Z_{\geq 0}$,
where $\zeta'$ has no \'etale segment $(1,0)$.
If there exists a specialization from $H(\xi')$ to $H(\zeta')$,
then considering the direct sum of it and $H(f(1,0))$
we have a specialization from $H(\xi)$ to $H(\zeta)$.

Assume that $\zeta$ has no \'etale segment $(0,1)$.
Let $k$ be an algebraically closed field.
Theorem~\ref{theorem1} in particular says that
there exists a ${\rm DM_1}$ ${\mathcal N}$ over $R=k[\![t]\!]$ such that
${\mathcal N}_k$ is isomorphic to $N_\zeta$ and ${\mathcal N}_K$
is isomorphic to $N_\xi$ over an algebraically closed field.
There exists a lifting of ${\rm DM_1}$ to a display $\mathcal{M}$ over $R$ 
(cf. \cite{harashita2}, Lemma 4.1).
Thanks to the theory of display by Zink \cite[Theorem 103 on p.\ 221]{Display},
we get the $p$-divisible group $\mathcal{X}$ associated with $\mathcal{M}$,
which is a specialization from $H(\xi)$ to $H(\zeta)$.
Indeed by \cite{oort2} the special fiber $\mathcal{X}_k$ is isomorphic to $H(\zeta)$
and the generic fiber $\mathcal{X}_K$ is isomorphic to $H(\xi)$
over an algebraically closed field.
\end{proof}


Here is another corollary.
Consider a $p$-divisible group $X$ over $k$.
Let $\operatorname{Def}(X) = \operatorname{Spf}(\Gamma)$ be the local deformation space in characteristic $p$ (cf. \cite[1.9]{oort3}).
Set $D(X) = \operatorname{Spec}(\Gamma)$
and let $\mathcal{X} \to D(X)$ be the $p$-divisible group
over $D(X)$ obtained via the category equivalence obtained in \cite[2.4.4]{Jong}.
For a $p$-divisible group $\mathcal{Y} \to S$ over a scheme $S$
and for a $p$-divisible group $Y$ over $k$,
we consider
\[
\mathcal{C}_{Y}(S) = \{s\in S \mid \mathcal{Y}_s \text{ is isomorphic to } Y
\text{ over an algebraically closed field}\}.
\]
To the dimension formula of $\mathcal{C}_{X}(D(X))$,
Oort gave three proofs in \cite{oort3}.
This paper gives the fourth proof.

\begin{corollary}\label{CorDimensionFormula}
Let $X$ be a $p$-divisible group of Newton polygon $\xi$ of height $h$ and of dimension $d$, we have
\[
\dim \mathcal{C}_X(D(X)) = c(\xi),
\]
where we set $c(\xi):=c(\sigma,\xi)$ with $\sigma = (h-d,d)$.
\end{corollary}
\begin{proof}
Let ${\mathcal S}_{Y[p]}(D(Y))$ be
the Ekedahl-Oort stratum on $D(Y)$, i.e., the locally closed subset
consisting of points whose $p$-kernel types are the same as that of $Y[p]$
(cf. \cite{oort3}, 1.6).
We consider it as a locally closed subscheme of $D(Y)$
by giving it the induced reduced structure.
By Oort \cite{oort3}, 7.19 and 7.9, we have
\[
\dim \mathcal{C}_X(D(X))=\dim \mathcal{C}_{H(\xi)}(D(H(\xi))) = \dim {\mathcal S}_{H(\xi)[p]}(D(H(\xi))).
\] 
Combining Wedhorn \cite[6.10]{wedhorn} and Moonen \cite[2.1.2]{moonen},
$\dim {\mathcal S}_{Y[p]}(D(Y))$
this is equal to the length $\ell(A)$ of the type $A$ of $Y[p]$, where
$\ell(A)=\# \{(i,j) \mid i<j,\ \delta_i=0,\ \delta_j=1\}$
for $A=\delta_1\cdots\delta_h$.
Let $\chi = (h-d)(1,0)+d(0,1)$.
By Theorem \ref{theorem1}, we have 
\[
\ell(A_\sigma)+c(\sigma,\xi)\le \ell(A_\xi) \le \ell(A_\chi)-c(\xi,\chi).
\]
The corollary follows from the obvious cases $\ell(A_\sigma)=0=c(\sigma)$ and $\ell(A_\chi) = (h-d)d = c(\chi)$, using the formula $c(\zeta,\xi) = c(\xi)-c(\zeta)$.
\end{proof}

\begin{remark}
In the proof of Corollary \ref{CorDimensionFormula},
we used the dimension formula of Ekedahl-Oort strata
due to Wedhorn \cite[6.10]{wedhorn} and Moonen \cite[2.1.2]{moonen},
but, as type $A$ in the proof is combinatorially complicated
(no explicit general form of $A$ is known), determinining $\ell(A)$ is a non-trivial problem. 
\end{remark}

\section{Proof}
In this section, we prove Theorem~\ref{theorem1}.
The essential case is that $\xi$ consists of two segments, i.e.,
$\xi = (m_1, n_1) + (m_2, n_2)$ with $\gcd(m_1, n_1)=1$ and $\gcd(m_2, n_2)=1$.
Here we treat the case of $\lambda_2 < 1/2 < \lambda_1$,
where $\lambda_i = n_i/(m_i+n_i)$ for $i=1,2$,
since the other case $\lambda_2<\lambda_1 \le 1/2$ or 
$1/2\le \lambda_2 < \lambda_1$ has already been studied in \cite{harashita1}, 8.4.

Let $\xi = (m_1, n_1) + (m_2, n_2)$ with $\lambda_2 < 1/2 < \lambda_1$,
until {\it Proof of Theorem \ref{theorem1}}.
Let $N_\xi$ be the minimal ${\rm DM_1}$ of $\xi$,
and we denote the sequence associated with $N_\xi$ as $A_\xi$.
In Lemma~\ref{lem1}, we have seen that the sequence $A_\xi$ is described as
$$\underbrace{1^A_1\cdots 1^A_{m_1}}_{m_1}
\underbrace{0^A_{m_1+1}\cdots 0^A_{n_1}}_{n_1-m_1}
\underbrace{1^B_1\cdots 1^B_{n_2}}_{n_2}\ 
\underbrace{0^A_{n_1+1}\cdots 0^A_{m_1+n_1}}_{m_1}
\underbrace{1^B_{n_2+1}\cdots 1^B_{m_2}}_{m_2-n_2}\ 
\underbrace{0^B_{m_2+1}\cdots 0^B_{m_2+n_2}}_{n_2}.$$
For a sequence $S$ of $0$ and $1$, let $S^-$ denote the sequence obtained by
exchanging ``first'' adjacent subsequence ``$0\ 1$'' for ``$1\ 0$'' in $S$.
To simplify, we write $(A_\xi^-)^- = A_\xi^{--}$.
Note $A_\xi^{--} < A_\xi^- < A_\xi$.
Let $N^{-}_{\xi}$ and $N^{--}_{\xi}$ be the ${\rm DM_1}$'s associated with
$A^{-}_{\xi}$ and $A^{--}_{\xi}$ respectively.
We shall prove Theorem~\ref{theorem1} by induction, using
the following three propositions.
\begin{proposition}\label{prop1}
If 
one of the following conditions
\begin{itemize}
\item[(1)] $n_1-m_1 > 1$ and $n_2 > 0$,
\item[(2)] $n_1 - m_1 = 1$ with $m_1 > 0$ and $n_2 = 1$,
\item[(3)] $n_2 = 0$ and $m_1 > 1$
\end{itemize}
holds, then we have
$$N^{--}_{\xi}=N^-_{\xi'}\oplus N_{\rho},$$
with $\rho=(a,b)$ and $\xi'=(m_1-a,n_1-b)+(m_2,n_2)$,
where $\rho$ is uniquely determined by $\xi$ so that the area of
the region surrounded by $\xi$, $\xi'$ and $\rho$ in the picture below is one.
\\
\\
\includegraphics[width=200mm]{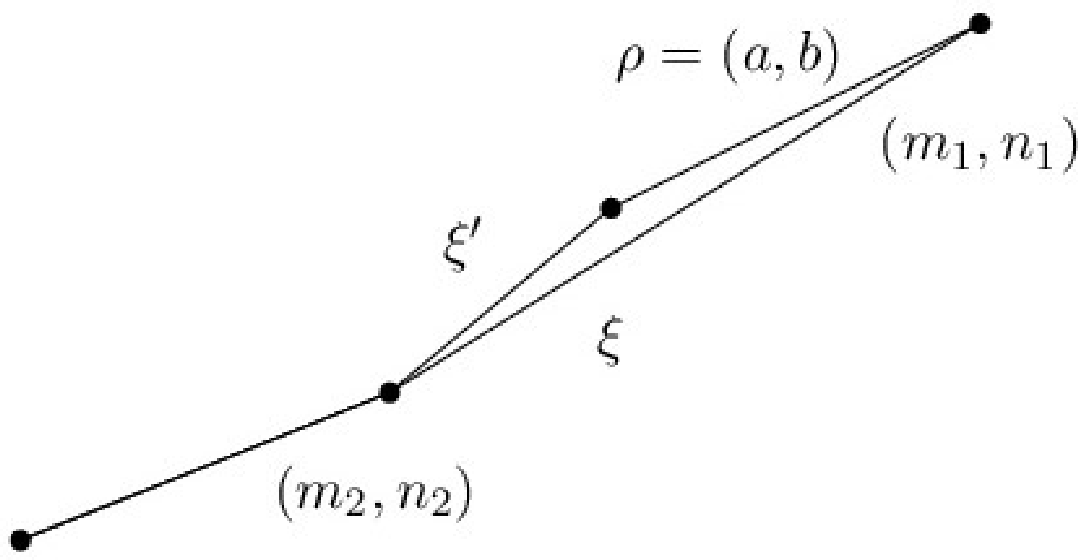}
\end{proposition}

\begin{proof}
First, let us see the case (1).
Since $\gcd(m_1, n_1) = 1$,
we have $m_1 > 0$.
The sequences $A = A_{m_1,n_1}$ and $B = A_{m_2,n_2}$
associated with $N_{m_1,n_1}$ and $N_{m_2,n_2}$ are
\begin{eqnarray*}
A &=& \underbrace{1^A_1\cdots 1^A_{m_1}}_{m_1}\ \underbrace{0^A_{m_1+1}\cdots 0^A_{n_1}}_{n_1-m_1}\ \underbrace{0^A_{n_1+1}\cdots 0^A_{m_1+n_1}}_{m_1},\\
B &=& \underbrace{1^B_1\cdots 1^B_{n_2}}_{n_2}\ \underbrace{1^B_{n_2+1}\cdots 1^B_{m_2}}_{m_2-n_2}\ \underbrace{0^B_{m_2+1}\cdots 0^B_{m_2+n_2}}_{n_2}
\end{eqnarray*}
respectively.
By Lemma~\ref{lem1}, the sequence $A \oplus B$
of $N_{\xi}=N_{m_1,n_1}\oplus N_{m_2,n_2}$ is
\\
\\
$$N_{\xi} : \def\objectstyle{\scriptstyle}\xymatrix@
=0.000001pt{1^A_1\ar@/_20pt/[rrrrrrrrr] &\cdots &1^A_{m_1} &0^A_{m_1+1}\ar@/_20pt/[lll] &\cdots &0^A_{n_1} &1^B_1\ar@/_20pt/[rrrrrr] &\cdots &1^B_{n_2} &0^A_{n_1+1} &\cdots &0^A_{m_1+n_1}\ar@/_20pt/[llllll] &1^B_{n_2+1} &\cdots &1^B_{m_2} &0^B_{m_2+1}\ar@/_20pt/[lllllllll] &\cdots &0^B_{m_2+n_2}\ar@/_20pt/[lllllllll]}.$$ 
Here, we write only arrows of $F$ and $V^{-1}$
which are necessary to check the structure of $N_{\xi}$.
Of course the $(F,V^{-1})$-diagram of $A\oplus B$ consists of 
the cycles of $A$ and $B$:
$$\xymatrix{\ \ar@{}[d]|{\normalsize{\mbox{$N_{m_1, n_1} :$}}} 
& 0^A_{m_1+1} \ar[d]&\cdots \ar[l] &0^A_{n_1}\ar[l]  
& &\ar@{}[d]|{\normalsize{\mbox{$N_{m_2, n_2} :$}}}  &0^B_{m_2+1} \ar[d]
&\cdots \ar[l] &1^B_{n_2} \ar[l]
\\ \ &1^A_1 \ar[r]  &\cdots \ar[r] &0^A_{h_1}, \ar[u]& & &1^B_1 \ar[r]
&\cdots \ar[r] &0^B_{h_2},\ar[u]}$$
where $h_1 = m_1 + n_1$ and $h_2 = m_2 + n_2$. 
The ${\rm DM_1}$ $N^-_{\xi}$ is obtained by exchanging 
$0^A_{n_1}$ and $1^B_1$ in $A_{\xi}$.
The sequence $A^-_{\xi}$ corresponding to $N^-_{\xi}$ is described as
\[
1^A_1\cdots 1^A_{m_1}\ 0^A_{m_1+1}\cdots 
0^A_{n_1-1}\ \underline{1^B_1\ 0^A_{n_1}}\ 1^B_2\cdots 1^B_{n_2}\ 0^A_{n_1+1}
\cdots 0^A_{h_1}\ 1^B_{n_2+1}\cdots 1^B_{m_2}\ 0^B_{m_2+1}
\cdots 0^B_{h_2}.
\]
We claim that $N^-_{\xi}$ consists of one cycle.
Indeed, by the exchange of $0^A_{n_1}$ and $1^B_1$,
the destination of the vector from 
$0^B_{m_2+1}$ is changed to $0^A_{n_1}$,
and the destination of the vector from 
$0^A_{m_1+n_1}$ is changed to $1^B_1$.
The claim follows from the diagram of $N_\xi^-$:
\begin{eqnarray*}
\xymatrix{& \ar@{}[l] & \ar@{}[d]|{\normalsize{\mbox{$N^-_\xi :$}}}
& 0^A_{m_1+1} \ar[d]&\cdots \ar[l] &0^A_{n_1}\ar[l] 
&0^B_{m_2+1} \ar@{.>}[d] |\times \ar[l]
&\cdots \ar[l] &1^B_{n_2} \ar[l]  & \ar@{}[l] & \ar@{}[l] 
\\ & & \ar@{}[l] & 1^A_1 \ar[r]  &\cdots \ar[r] 
&0^A_{h_1} \ar@{.>}[u] |\times \ar[r] &1^B_1 \ar[r]
&\cdots \ar[r] &0^B_{h_2}.\ar[u] \ar@{}[r] & \ar@{}[r] 
& \ar@{}[u]^{\normalsize{\mbox{($\star$)}}} } 
\end{eqnarray*}
We get the ${\rm DM_1}$ $N^{--}_{\xi}$ by exchanging 
$0^A_{n_1-1}$ and $1^B_1$ in $A^-_{\xi}$.
The sequence $A^{--}_{\xi}$ corresponding to $N^{--}_{\xi}$
is described as

$$\def\objectstyle{\scriptstyle}\xymatrix@
=0.000001pt{1^A_1\ar@/_20pt/[rrrrrrrrrrrr] &\cdots &1^A_{m_1} &0^A_{m_1+1}\ar@/_20pt/[lll] &\cdots &0^A_{n_1-2} &1^B_1\ar@/_20pt/[rrrrrrrr] &0^A_{n_1-1} &0^A_{n_1} &1^B_2 &\cdots &1^B_{n_2} &0^A_{n_1+1} \cdots &0^A_{h_1}\ar@/_20pt/[llllll] &1^B_{n_2+1} &\cdots &1^B_{m_2} &0^B_{m_2+1}\ar@/_20pt/[lllllllll] &\cdots&0^B_{h_2}\ar@/_20pt/[llllllll]}.$$\\
We shall use these arrows to check the structure of cycles of $N^{--}_{\xi}$.
To consider the cycles of $N^{--}_{\xi}$, 
we rewrite the $(F,V^{-1})$-cycle of $N_\xi^-$ as follows:
$$\xymatrix{\ar@{}[d]|{\normalsize{\mbox{$N^-_\xi :$}}} 
&\bullet \ar[d] &\cdots \ar[l] &0^A_{n_1-1} \ar[l] 
&0^A_{h_1-1} \ar[l] &\cdots \ar[l] &0^A_{n_1} \ar[l]
\\&\bullet \ar[r] &\cdots \ar[r] &0^A_{h_1} \ar[r] 
&1^B_1 \ar[r] &\cdots \ar[r] &0^B_{m_2+1}. \ar[u]}$$
We claim that $N^{--}_{\xi}$ consists of two cycles. 
By the exchange of $0^A_{n_1-1}$ and $1^B_1$,
the destination of the vector from $0^A_{h_1}$
is changed to $0^A_{n_1-1}$, and the destination of the vector from
$0^A_{h_1-1}$ is changed to $1^B_1$. The claim follows from the diagram:\\
$$\xymatrix{\ar@{}[d]|{\normalsize{\mbox{$N^{--}_\xi :$}}} 
&\bullet \ar[d]&\cdots \ar[l] &0^A_{n_1-1}\ar[l] 
&0^A_{h_1-1} \ar@{.>}[l] |\times \ar[d] 
&\cdots \ar[l] &0^A_{n_1} \ar[l]
\\ &\bullet \ar[r]  &\cdots \ar[r] &0^A_{h_1} \ar@{.>}[r] |\times \ar[u] &1^B_1 \ar[r]
&\cdots \ar[r] &0^B_{m_2+1}. \ar[u]}$$
Let $C_1$ be the cycle containing $0^A_{h_1}$, 
and let $C_2$ be the cycle containing $1^B_1$.
It suffices to show the following properties:

\begin{enumerate}
\item[(a)] $C_1$ is the $(F,V^{-1})$-diagram of $N_\rho$,
\item[(b)] $C_2$ is the $(F,V^{-1})$-diagram of $N^-_{\xi'}$.
\end{enumerate}
As $C_1$ is a part of the cycle of simple ${\rm DM_1}$ $N_{m_1,n_1}$,
the sequence corresponding to $C_1$ is written as follows:
$$\underbrace{1^A\cdots 1^A}_{a}\ \underbrace{0^A\cdots 0^A_{m_1+n_1}}_{b}$$
Since this cycle coincides with the cycle obtained from $A$ by
applying \cite{harashita1}, Lemma 5.6 to 
the adjacent $0_{n_1-1}^A\ 0_{n_1}^A$,
we have $a n_1 - b m_1 = 1$.
Hence the property (a) holds.

Next we consider the cycle $C_2$.
We can write $C_2$ as follows.\\ 
\\
$$\def\objectstyle{\scriptstyle}
\xymatrix@
=0.000001pt{1^A\ar@/_30pt/[rrrrrrrrrrr] &\cdots
&1^A\ar@/_30pt/[rrrrrrrrrrr]  &0^A \ar@/_30pt/[lll]
&\cdots &0^A &1^B_1\ar@/_30pt/[rrrrrrrrr] 
&0^A_{n_1} &1^B &\cdots &1^B &0^A &\cdots &0^A &0^A_{h_1-1}\ar@/_30pt/[llllllll] &1^B_{n_2+1} &\cdots &1^B\ar@/_30pt/[rrrr] &0^B_{m_2+1}\ar@/_30pt/[lllllllllll] 
&0^B\ar@/_30pt/[lllllllllll] &\cdots &0^B_{h_2}\ar@/_30pt/[lllllllllll]}$$ \\ 
\\
Let $L$ be the sequence associated to $C_2$,
and let $L'$ be the sequence constructed by exchanging 
$1^B_1$ and $0^A_{n_1}$ in $L$. 
Then we have
$$L' = \underbrace{1^A\cdots 1^A}_{m_1-a}\underbrace{0^A\cdots 0^A_{n_1}}_{(n_1-b)-(m_1-a)} \underbrace{1^B_1\cdots 1^B}_{n_2}\ \underbrace{0^A\cdots 0^A}_{m_1-a} \ \underbrace{1^B\cdots 1^B}_{m_2-n_2}\ \underbrace{0^B\cdots 0^B}_{n_2}.$$
By Lemma~\ref{lem1}, 
we see that this sequence is decomposed into two simple ${\rm DM_1}$'s as follows:
$$\underbrace{1^A\cdots 1^A}_{m_1-a}\ \underbrace{0^A\cdots 0^A_{n_1}\ 0^A\cdots 0^A}_{n_1-b}\, 
\oplus\,  \underbrace{1^B_1\ 1^B\cdots 1^B}_{m_2}\ \underbrace{0^B\cdots 0^B}_{n_2}
=A_{m_1-a,n_1-b}\oplus A_{m_2,n_2}.$$
It is clear that $L'$ corresponds to $N_{\xi'}$
with $\xi'=(m_1-a,n_1-b)+(m_2,n_2)$, whence the property (b) holds.

Let us see the case (2).
In this case, for $A_\xi^-$ given by 
exchanging the subsequence $0^A_{n_1}\ 1^B_1$ for $1^B_1\ 0^A_{n_1}$,
we construct $A_\xi^{--}$ by the exchange of $0^A_{h_1}$ and $1^B_2$
with $h_1 = m_1 + n_1$.
The sequence $A_\xi^{--}$ consists of two cycles.
The cycle containing $0^A_{h_1}$ is associated with $N_{1,1}$.
In fact, the sequence of this cycle is $1^B_1\ 0^A_{h_1}$.
Moreover, for the other cycle, 
one can check that this cycle is associated with $N_{\xi'}^-$
with $\xi' = (m_1-1, n_1-1) + (m_2, n_2)$ 
by exchanging $1^A_{m_1}$ and $0^A_{n_1}$.

For the case (3), since $\gcd(m_2, n_2) = 1$, we have $m_2 = 1$.
In this case, we have only one subsequence ``$0\ 1$''
in $A_\xi$, which is $0^A_{h_1}\ 1^B_1$.
We construct $A_\xi^-$ by exchanging this subsequence for $1^B_1\ 0^A_{h_1}$.
We exchange $0^A_{h_1-1}$ and $1^B_1$ to construct $A_\xi^{--}$.
Then $A_\xi^{--}$ consists of two cycles.
The cycle containing $0^A_{h_1-1}$ is associated with $N_{a,b}$
as this cycle consists of some terms of $A$.
Since this cycle coincides with the cycle obtained from $A$ by
applying \cite{harashita1}, Lemma~5.6 to 
the adjacent $0^A_{h_1-1}\ 0^A_{h_1}$, 
we see $an_1 - bm_1 = 1$.
One can check that the other cycle is associated with $N_{\xi'}^-$
with $\xi' = (m_1-a, n_1-b) + (m_2, n_2)$
by exchanging $1^B_1$ and $0^A_{h_1}$.
\end{proof}

\begin{example}\label{example2}
We have dealt with the case of $\xi=(3,5)+(3,2)$ in Example \ref{example1},
where we observed that
\[
N_{(3,5)+(3,2)}^{--} = N_{(1,2)+(3,2)}^- \oplus N_{2,3}
\]
holds.

\end{example}

\begin{proposition}\label{prop2}
If $n_1 - m_1 = 1$ and $n_2 > 1$,
then we have
$$N^{--}_{\xi}=N^-_{\xi'}\oplus N_{\rho},$$
with $\rho=(a,b)$ and $\xi' = (m_1, n_1) + (m_2-a, n_2-b)$,
where $\rho$ is uniquely determined by $\xi$ so that 
the area of the region surrounded by $\xi$, $\xi'$ and $\rho$
in the picture below
is one.\\
\\
\includegraphics[width=200mm]{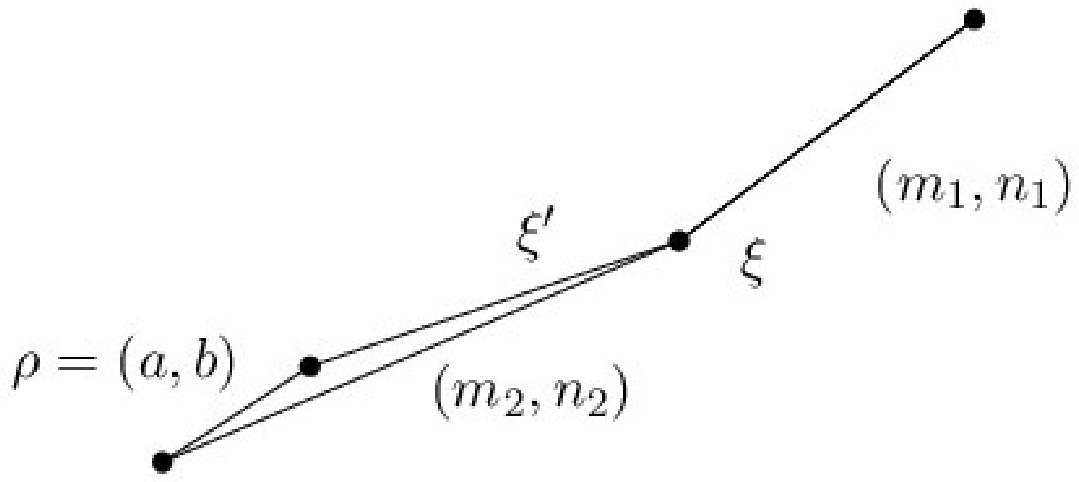}
\end{proposition}

\begin{proof}
In this hypothesis, we have $0^A_{m_1+1} = 0^A_{n_1}$.
The sequence $A_\xi^-$ is obtained by exchanging $0^A_{n_1}$ and $1^B_1$.
Then we obtain the same $(F, V^{-1})$-diagram as ($\star$). 
For the sequence $A_\xi^-$:
\[
1^A_1\cdots 1^A_{m_1}\ 1^B_1\ 
\underline{0^A_{n_1}\ 1^B_2}\cdots 1^B_{n_2}\ 0^A_{n_1+1}
\cdots 0^A_{m_1+n_1}\ 1^B_{n_2+1}\cdots 1^B_{m_2}\ 0^B_{m_2+1}
\cdots 0^B_{m_2+n_2},
\]
to construct $A_\xi^{--}$, we exchange $0^A_{n_1}$ and $1^B_2$.
Then $A_\xi^{--}$ consists of two cycles.
The ${\rm DM_1}$ $N_{\xi}^{--}$ is described as follows:\\
$$\xymatrix{\ar@{}[d]|{\normalsize{\mbox{$N^{--}_\xi :$}}}
&\bullet \ar[d]
&\cdots \ar[l] 
&0^A_{n_1}\ar[l] 
&0^B_{m_2+1} \ar@{.>}[l] |\times \ar[d] 
&\cdots \ar[l] 
&1^B_{n_2} \ar[l] \\ 
&\bullet \ar[r]  
&\cdots \ar[r] 
&0^B_{m_2+2} \ar@{.>}[r] |\times \ar[u] 
&1^B_2 \ar[r]
&\cdots \ar[r] 
&0^B_{m_2+n_2}.\ar[u]}$$
Let $C_1$ (resp. $C_2$) be the cycle of $A_\xi^{--}$
containing $0^A_{n_1}$ (resp. $1^B_2$).
Since the cycle $C_2$ consists of some terms of $B$,
this cycle corresponds to a minimal ${\rm DM_1}$ 
$N_\rho$ with $\rho = (a, b)$.
Since this cycle coincides with the cycle obtained from $B$ 
the adjacent $1^B_1\ 1^B_2$, by
\cite{harashita1}, Lemma~5.6,
we have $b m_2 - an_2= 1$.
Let us see that $C_1$ corresponds to $N_{\xi'}^-$
for a Newton polygon $\xi'$.
Let $L$ be the sequence corresponding to $C_1$,
and let $L'$ be the sequence constructed by exchanging 
$1^B_1$ and $0^A_{n_1}$ in $L$. 
Then we have
$$L' = \underbrace{1^A_1\cdots 1^A_{m_1}}_{m_1}\ 
0^A_{n_1}\ \underbrace{1^B_1\cdots 1^B}_{n_2-b}\ 
\underbrace{0^A_{n_1+1}\cdots 0^A_{m_1+n_1}}_{m_1}\ 
\underbrace{1^B\cdots 1^B}_{(m_2-a)-(n_2-b)}\ 
\underbrace{0^B\cdots 0^B}_{n_2-b}.$$
It implies that $L'$ is associated with $N_{\xi'}$ 
for $\xi' = (m_1, n_1) + (m_2-a, n_2-b)$.
\end{proof}

Let us see an example of this specialization.

\begin{example}
We consider $N_{\xi}=N_{2,3}\oplus N_{4,3}$.
By Lemma~\ref{lem1}, the sequence of  $N_{\xi}$ is
$$A_{\xi}=1^A_1\ \ 1^A_2\ \ 0^A_3\ \ 1^B_1\ \ 1^B_2\ \ 
1^B_3\ \ 0^A_4\ \ 0^A_5\ \ 1^B_4\ \ 0^B_5\ \ 0^B_6\ \ 0^B_7.$$
By exchanging $0^A_3\ 1^B_1$ for $1^B_1\ 0^A_3$, we obtain $A_\xi^-$.
Moreover, by exchanging $0^A_3\ 1^B_2$ for $1^B_2\ 0^A_3$,
we obtain $A_\xi^{--}$.
The sequence associated with $N_\xi^{--}$ can be decomposed 
into two cycles as follows.\\ 
\\
\begin{equation}\label{dgm}
N^{--}_{\xi} : \xymatrix@=10pt{1\ar@/_20pt/[rrrrr] &1\ar@/_20pt/[rrrrr] 
&1\ar@/_20pt/[rrrrr] &0\ar@/_20pt/[lll] &1\ar@/_20pt/[rrrr] 
&0\ar@/_20pt/[llll] &0\ar@/_20pt/[llll] &1\ar@/_20pt/[rr] &0\ar@/_20pt/[lllll]
&0\ar@/_20pt/[lllll]}\ 
\oplus \ 
\xymatrix@=10pt{1\ar@/_20pt/[r] &0\ar@/_20pt/[l]} \tag{a}
\end{equation}
\\
Let $\xi'=(2,3)+(3,2)$. 
We have then
$A_{\xi'}=1\ 1\ 0\ 1\ 1\ 0\ 0\ 1\ 0\ 0$ and
$A^-_{\xi'}=1\ 1\ 1\ 0\ 1\ 0\ 0\ 1\ 0\ 0$.
One verifies that $A^-_{\xi'}$ coincides with the left direct summand of \eqref{dgm}.
\end{example}

The final proposition treats the remaining case:

\begin{proposition}\label{prop3}
We have
\begin{itemize}
\item[(1)] If $m_1 = n_2 = 0$, then $N^-_{\xi}=N_{1,1}$;
\item[(2)] If $m_1 = 0$ and $n_2 = 1$, then $N^{--}_\xi = N_{1,1} \oplus N_{m_2-1, 1}$;
\item[(3)] If $m_1 = 1$ and $n_2 = 0$, then $N^{--}_\xi = N_{1, n_1-1} \oplus N_{1,1}$.
\end{itemize}
\end{proposition}

\begin{proof}
In the case (1), $N_{\xi}$ is expressed by
\\
$$N_{\xi} : \xymatrix@=10pt{0\ar@(ur,ul) & 1\ar@(dl,dr)}.$$
\\
We construct $N^-_{\xi}$ by exchanging 0 and 1 in $N_{\xi}$, 
then we get
$$N^-_{\xi} : \xymatrix@=10pt{1\ar@/_15pt/[r] &0\ar@/_15pt/[l]}.$$
Hence $N^-_{\xi}=N_{1,1}$. 

In the case (2), we exchange the subsequence $0^A_1\ 1^B_1$
to construct $A_\xi^-$.
Moreover, we construct $A_\xi^{--}$ by exchanging 
the subsequence $0^A_1\ 1^B_2$ for $1^B_2\ 0^A_1$.
Then $A_\xi^{--}$ consists of two cycles.
The cycle containing $0^A_1$ is associated with $N_{1,1}$
since the sequence of this cycle is $1^B_1\ 0^A_1$.
Moreover, the other cycle is associated with $N_{m_2-1, 1}$.

In the case (3), The sequence $A_\xi^-$ is constructed by
exchanging $0^A_{h_1}\ 1^B_1$ for $1^B_1\ 0^A_{h_1}$
with $h_1 = m_1+n_1$.
We exchange the subsequence $0^A_{h_1-1}\ 1^B_1$
for $1^B_1\ 0^A_{h_1-1}$, and we construct $A_\xi^{--}$.
Then $A_\xi^{--}$ consists of two cycles.
The cycle containing $1^B_1$ is associated with $N_{1,1}$.
In fact, the sequence of this cycle is $1^B_1\ 0^A_{h_1}$.
The other cycle is associated with $N_{1, n_1-1}$.
This completes the proof.
\end{proof}

Finally, we see that Propositions~\ref{prop1}, \ref{prop2} and \ref{prop3}
imply Theorem~\ref{theorem1}.

\begin{proof}[Proof of Theorem~\ref{theorem1}]
It suffices to show the essential case that
$\zeta\prec\xi$ is saturated and $\xi$ consists of two segments.
Put $\xi=(m_1,n_1)+(m_2,n_2)$ and assume $\zeta\prec \xi$ is saturated.
Set $\lambda_1=n_1/(m_1+n_1)$ 
and $\lambda_2=n_2/(m_2+n_2)$.
The case that the slopes of $\xi$, say $\lambda_2 < \lambda_1$, 
satisfy $\lambda_2<\lambda_1 \leq 1/2$ and 
$1/2 \leq \lambda_2 < \lambda_1$
has been proved in \cite{harashita1}, 8.4.
As it suffices to prove the remaining case,
we assume $\lambda_2<1/2<\lambda_1$.
We claim that
there exist $A^{(1)}, \ldots ,A^{(c-2)}$ with $c=c(\zeta,\xi)$
such that $A_{\zeta} =: A^{(0)} < \cdots < A^{(c-2)} < A^{(c-1)}:=A^-_{\xi}$.
This claim implies the theorem by $A^-_{\xi} < A_{\xi}$.
Recall that for a sequence $S$ we denote by $S^-$ the sequence obtained by exchanging the first adjacent subsequence ``$0\ 1$" in $S$ for ``$1\ 0$",
and $S^{--}$ means $(S^-)^-$.
We use induction on $h=m_1+n_1+m_2+n_2$.
The claim is clear
when $h$ is smallest, i.e., $h=2$ with
$\zeta=(1,1) \prec \xi=(0,1)+(1,0)$,
see the case of $m_1=n_2=0$ in Proposition \ref{prop3}. 
Assume $h>2$.
Note that $\xi$ with $h>2$ belongs to either of the following six cases.
\begin{enumerate}
\item[(i)] $n_1-m_1 > 1$ and $n_2 > 0$,
\item[(ii)] $n_1 - m_1 = 1$ with $m_1 > 0$ and $n_2 = 1$,
\item[(iii)] $n_2 = 0$ and $m_1 > 1$,
\item[(iv)] $n_1 - m_1 = 1$ and $n_2 > 1$,
\item[(v)] $m_1 = 0$ and $n_2 = 1$,
\item[(vi)] $m_1 = 1$ and $n_2 = 0$.
\end{enumerate}
In the case of (i), (ii) and (iii), 
we get $N^{--}_{\xi}=N^-_{\xi'}\oplus N_{\rho}$ by Proposition~\ref{prop1},
where $\zeta = \zeta' + \rho$ and $\zeta' \prec \xi'$ is saturated.
In the case of (iv), 
we get $N^{--}_{\xi}=N^-_{\xi'}\oplus N_{\rho}$ by Proposition~\ref{prop2},
where $\zeta = \zeta' + \rho$ and $\zeta' \prec \xi'$ is saturated.
By the hypothesis of induction, there exist
$B^{(1)}, \ldots, B^{(c-3)}$ such that
$A_{\zeta'}=:B^{(0)} < \cdots < B^{(c-3)} < B^{(c-2)}:= A_{\xi'}^-$.
Here we note that $c(\zeta', \xi') = c(\zeta, \xi)-1$
holds by Proposition~\ref{prop4}.
Hence putting $A^{(i)} = B^{(i)} \oplus A_{\rho}$ ($i=0,\ldots, c-2$), we have
\[
A_{\zeta} = A_{\zeta'}\oplus A_{\rho} 
= A^{(0)} < \cdots < A^{(c-2)} 
= A^-_{\xi'}\oplus A_{\rho} = A_\xi^{--} < A_\xi^-.
\]
In the case of (v) and (vi), 
we get $N^{--}_{\xi}=N_{\rho} \oplus N_{\rho'}$ by Proposition~\ref{prop3},
where $\zeta = \rho + \rho'$ and $\zeta \prec \xi$ is saturated,
whence we have $A_\zeta = A_\xi^{--} < A_\xi^-$.
Thus we have proved the claim for all the cases.
\end{proof}

\subsection*{Acknowledgments}
The authors thank the referee for careful reading and helpful comments.
This research is supported by
JSPS Grant-in-Aid for Scientific Research (C) 17K05196.


\begin{thebibliography}{9}

\bibitem{harashita1}
S. Harashita: 
{\it Configuration of the central streams in the moduli of abelian varieties}. Asian J. Math. {\bf 13} (2009), no. 2, 215--250.

\bibitem{harashita2}
S. Harashita:
{\it The supremum of Newton polygons of $p$-divisible groups
with a given $p$-kernel type}.
Geometry and Analysis of Automorphic Forms of Several Variables,
Proceedings of the international symposium in honor of Takayuki Oda
on the occasion of his 60th birthday, Series
on Number Theory and Its Applications, Vol. {\bf 7}, (2011), pp. 41-55.

\bibitem{SNP}
S. Harashita: 
{\it On $p$-divisible groups with saturated Newton polygons}. 
Nagoya Mathematical Journal, 1-25. doi: 10.1017/nmj.2017.22.

\bibitem{Jong}
A. J. de Jong: 
{\it Crystalline Dieudonne module theory via formal and rigid geometry}. 
Inst. Hautes \'Etudes Sci. Publ. Math.  No. {\bf 82}  (1995), 5--96 (1996).

\bibitem{Katz}
N. M. Katz: 
{\it Slope filtration of $F$-crystals}.
Journ. G\'eom. Alg. Rennes, Vol. I, Ast\'erisque {\bf 63} (1979),
Soc. Math. France, 113--164.

\bibitem{kraft}
H. Kraft:
{\it Kommutative algebraische $p$-Gruppen}. 
Sonderforschungsbereich. Bonn, September 1975. Ms. 86 pp.

\bibitem{Manin} Yu. I. Manin:
{\it Theory of commutative formal groups over fields of finite characteristic}.
Uspehi Mat. Nauk {\bf 18} no.~6 (114) (1963), pp.~3-90.

\bibitem{moonen}
B. Moonen:
{\it A dimension formula for Ekedahl-Oort strata}.
Ann. Inst. Fourier {\bf 54} (2004), 666--698.

\bibitem{moonen-wedhorn}
B. Moonen and T. Wedhorn: 
{\it Discrete invariants of varieties in positive characteristic}. 
Int. Math. Res. Not. 2004, no. {\bf 72}, 3855--3903.

\bibitem{oort1}
F. Oort: 
{\it A stratification of a moduli space of abelian varieties}. 
in: Moduli of Abelian Varieties (Ed. C. Faber, G. van der Geer, F. Oort), Progr. Math., {\bf 195}, Birkh\"auser, Basel, 2001; pp. 345-416.

\bibitem{oort}
F. Oort: 
{\it Foliations in moduli spaces of abelian varieties}. 
J. Amer. Math. Soc. {\bf 17} (2004), no.2, 267-296. 

\bibitem{oort2}
F. Oort: 
{\it Minimal $p$-divisible groupe}. 
Ann. of Math. (2) {\bf 161} (2005), 1021--1036.

\bibitem{oort3}
F. Oort: 
{\it Foliations in moduli spaces of abelian varieties and dimension of leave}. 
Progr. Math. {\bf 270}, Birkh\"{a}user Boston, Inc., Boston, MA, 2009.


\bibitem{VW}
E. Viehmann and T. Wedhorn:
{\it Ekedahl-Oort and Newton strata for Shimura varieties of PEL type}.
Math. Ann. {\bf 356} (2013), 1493--1550. 

\bibitem{wedhorn}
T. Wedhorn:
{\it The dimension of Oort strata of Shimura varieties of PEL-type}.
Moduli of abelian varieties. (Ed. C. Faber, G. van der Geer, F. Oort).
Progress in Math. {\bf 195}, Birkh\"auser Verlag 2001; pp. 441--471.

\bibitem{Display} 
Th. Zink:
{\it The display of a formal $p$-divisible group}.
Cohomologies $p$-adiques et applications arithm\'etiques, I.
	Ast\'erisque  No. {\bf 278}  (2002), 127--248.
\end{thebibliography}
\end{document}